\documentclass[12pt]{amsart}
\usepackage{amsmath,amssymb,amsfonts,amsthm,longtable}
\usepackage{amscd,pstricks,xypic}
\usepackage[dvips]{graphicx}

\newtheorem{thm}{Theorem }[section]

\newtheorem{lemma}[thm]{Lemma }

\newtheorem{prop}[thm]{Proposition }

\theoremstyle{definition}

\newtheorem{rem}[thm]{Remark }

\def\QQ{{\mathbb Q}}
\def\ZZ{{\mathbb Z}}

\def\FF{{\mathbb F}}
\def\fq{{\mathbb{F}_q}}
\def\kk{{\bar{k}}}

\def \bra#1\ket {\mathop{\vphantom{#1}\left<\smash{#1}\right>}\nolimits}

\DeclareMathOperator{\Hom}{Hom}
\DeclareMathOperator{\End}{End} 
\DeclareMathOperator{\Ker}{Ker} 
\DeclareMathOperator{\Mod}{mod} 
\DeclareMathOperator{\tr}{tr}

\DeclareMathOperator{\Spec}{Spec}
\DeclareMathOperator{\length}{length}

\def\Pr{\mathrm{Pr}}

\def\C{\mathcal{C}}
\def\B{\mathcal{B}}

\def\p{\mathfrak{p}}

\def\IK{\mathcal{IK}}
\def\OO{\mathcal{O}}
\def\P{\mathcal{P}}
\def\Fr{F}

\newcommand \eps {\varepsilon}
\renewcommand \phi {\varphi}

\begin{document}
\author{Sergey Rybakov}
\address{ Weizmann Institute of Science, Rehovot, Israel }

\email{rybakov.sergey@gmail.com}%
\title[Principal polarizations on abelian varieties]{Principal polarizations on products of abelian \\ varieties over finite fields}

\date{}


\begin{abstract}
We refine and generalize the results of K. E. Lauter and E. W. Howe on principal polarizations on products of abelian varieties over finite fields. Firstly, we study the reasons for the absence of an irreducible principal polarization in the isogeny class of the product of an ordinary and a supersingular abelian variety. Secondly, we provide a necessary condition for the existence of a principal polarization on an abelian variety in the isogeny class of the product of a geometrically simple abelian surface and an elliptic curve. As an application, we prove that this abelian threefold or its quadratic twist is a Jacobian.
\end{abstract}

\maketitle

\section{Introduction}
In this paper, we assume that the ground field $k=\fq$ is a finite field of characteristic $p$ and cardinality $q$. The Jacobian $J(C)$ of a smooth projective curve $C$ over $k$ can be non-simple as an abstract abelian variety without polarization. This is the case for many interesting curves, for example, for maximal or minimal curves over finite fields.
This observation is used by K. Lauter in~\cite{KL}, where she constructs genus $3$ curves over finite fields that have the maximum or minimum number of points. 

Her work relies on the following idea of J.-P. Serre~\cite{Ser83}. If a curve $C$ over $k$ has many points (in other words, the defect is small), then there is an isogeny \[\phi:A\times B\to J(C),\] where $A$ and $B$ are abelian varieties of positive dimension.
 The pullback to $A\times B$ of the canonical polarization on $J(C)$ is a product of 
polarizations $L_A$ and $L_B$ on $A$ and $B$ respectively. Moreover, \[\deg L_A=\deg L_B=\deg\phi,\] and there is an anti-isometry $\ker L_A\cong \ker L_B$. This construction can be inverted: we can \emph{glue} two polarizations $L_A$ and $L_B$ on $A$ and $B$, if 
there is an anti-isometry $\ker L_A\cong \ker L_B$ (see Section~\ref{S3}). 
In many interesting cases it is possible to prove that there is no curve over a given finite field with a given number of points using the fact that the Jacobian of such a curve has to be the gluing of polarizations on two abelian varieties of positive dimension.  

For example, in some situations there is no irreducible principal polarization
in the isogeny class of the product of an ordinary and a supersingular abelian variety.
In this paper, we study the reasons for this in more detail;
in particular, we reprove and clarify the theorem of K. Lauter and E. Howe~\cite[Theorem 3.1]{HL}. 

The gluing construction is also used for a complete classification of zeta functions of curves of genus $2$ obtained in~\cite{HNR06}, where the results of~\cite{Ru} and \cite{How95} are used. In this paper, we extend this approach to the curves of genus $3$. By~\cite[Theorem 1.2]{How96}, if an abelian variety of dimension $3$ is simple, then it always has a principal polarization. We are interested in a less general case where an abelian variety is isogenous to a product of a geometrically simple abelian surface and an elliptic curve. 

In this case, we encounter the phenomenon of exceptional prime numbers. In Section~\ref{ExPrimes} we study them in full generality, but for geometrically simple abelian surfaces, they can be defined using Lemma~\ref{lemma1}.(1) as follows.

Let $A$ be a geometrically simple abelian surface with the Weil polynomial $f_A$, and the real Weil polynomial $h$ (see section~\ref{intro} for the definition of Weil polynomials). The endomorphism algebra $K=\End^\circ(A)$ of $A$ is a CM-field, and the real subfield $K^+$ of $K$ is isomorphic to the quadratic extension $\QQ[t]/h(t)\QQ[t]$. 
We say that a prime $\ell\in\ZZ$ is \emph{exceptional} (with respect to $A$) if 
\begin{itemize}
\item $f_A(t)\equiv f(t)^2\bmod\ell^2$, for some irreducible $f\in\ZZ_\ell[t]$;
\item $\ell$ is inert in $K^+$.
\end{itemize}

Let $B$ be an elliptic curve over $k$.
According to Remark~\ref{ExRem}, if $\ell$ is exceptional, then for any isogeny $A'\to A$, and any polarization $L$ on $A'$ \[\ker L\not\cong B[\ell].\] However, by Proposition~\ref{prop1}, in this situation the gluing of polarizations is still possible, at least if $A$ is ordinary. That is, there exist isogeny $A'\to A$ and a polarization $L$ on $A'$ such that $\ker L\cong B[\ell^2]$.

If the Weil polynomial $f_B$ of the elliptic curve $B$ is irreducible, then $\End^\circ(B)$ is a quadratic extension of $\QQ$. Denote by $\Delta_B$ the discriminant of $\End^\circ(B)$.
We are now ready to formulate the main result of the paper.

\begin{thm}\label{thm_f}
Let $A$ be a geometrically simple abelian surface with real Weil polynomial $h$.
Suppose that $B$ is an elliptic curve with irreducible Weil polynomial $f_B(t)=t^2-bt+q$.
Assume that there is a prime $\ell$ that divides $h(b)$ and such that the following conditions hold:
\begin{itemize}
\item $\Delta_B\neq -\ell$;
\item if $f_B(t)\equiv(t-t_1)^2\bmod\ell$, then $\ell^2$ divides $f_B(t_1)$;
\item if $\ell$ is exceptional, then $A$ is ordinary.
\end{itemize}
Then $A\times B$ is isogenous to an abelian variety with irreducible principal polarization.

Vice versa, if $A\times B$ is isogenous to an abelian variety with irreducible principal polarization, then $h(b)\neq\pm 1$.
\end{thm}

To an abelian threefold with irreducible principal polarization one can apply
 the following result due to Serre based on the Torelli Theorem and~\cite{OU}. 

\begin{thm}\cite[Section 7]{KL1}\label{OUS}
Let $A$ be an abelian variety of dimension $3$ over $k$ with a principal polarization. 
Assume that the polarization is irreducible over an algebraic closure of $k$.
  Then $A$ or its quadratic twist {\rm(}corresponding to $-1$ automorphism{\rm)} is the Jacobian of a smooth curve over $k$.
\end{thm}

From Theorem~\ref{thm_f} and Theorem~\ref{OUS} we immediately get the following result.

\begin{thm}
\label{3curv} Let $A$ be a geometrically simple abelian surface with a real Weil polynomial $h$. 
Suppose that $B$ is an elliptic curve with irreducible Weil polynomial $f_B(t)=t^2-bt+q$. Assume that there is a prime $\ell$ that divides $h(b)$ and such that the following conditions hold:
\begin{itemize}
\item $\Delta_B\neq -\ell$;
\item if $f_B(t)\equiv(t-t_1)^2\bmod\ell$, then $\ell^2$ divides $f_B(t_1)$;
\item if $\ell$ is exceptional, then $A$ is ordinary.
\end{itemize}
  Then the isogeny class $A\times B$ or its quadratic twist {\rm(}corresponding to $-1$ automorphism{\rm)}	contains the Jacobian of a smooth curve of genus $3$. 
Vice versa, if $A\times B$ is isogenous to a Jacobian, then $h(b)\neq\pm 1$.
\end{thm}

{\bf Acknowledgments}
The author thanks the anonymous reviewer for many useful remarks on the paper.
The author is grateful to the Weizmann Institute Emergency Program to Host Visiting Scientists Affected by the War in Ukraine.

\section{Preliminaries}\label{intro}
\subsection{Endomorphism algebras of abelian varieties}
Let $A$ and $B$ be abelian varieties over $k$, and let $\Hom(A,B)$ be the group of
homomorphisms from $A$ to $B$ over $k$. The group $\Hom(A,B)$ is finitely generated and
torsion-free and $\End(A)=\Hom(A,A)$ has a ring structure with
composition as multiplication. We will use the following notation:
\[\Hom^{\circ}(A,B)=\Hom(A,B)\otimes_{\ZZ}\QQ\text{, and }
\End^{\circ}(A)=\End(A)\otimes_{\ZZ}\QQ.\] We will call an abelian
variety $A$ {\it simple} if it does not contain nontrivial
abelian subvarieties. We will call an abelian
variety $B$ {\it isotypic} if there exists a simple
abelian variety $A$ such that $B$ is isogenous to $A^r$ for some $r$.
Any abelian variety $A$ over $k$ is isogenous to a product of isotypic abelian varieties $A_i$.
This decomposition corresponds to the decomposition of $\End^{\circ}(A)$ into a product of simple algebras
$\End^{\circ}(A_i)$. In particular, $\End^{\circ}(A)$ is a
semi-simple $\QQ$-algebra~\cite[IV.19. Corollaries 1 and 2]{Mum}.

An element $\phi\in\End(A)$ is called an \emph{isogeny} if $\phi$ is finite and surjective.
The kernel $\ker\phi$ of an isogeny $\phi$ is a finite group scheme; the \emph{degree} $\deg\phi$ of an isogeny is defined to be the order of its kernel. Note that since the kernel could be non-reduced, the order of $\ker L$ is not the same as the order of its group of points $\ker L(\kk)$. If $\phi\in\End(A)$ is not an isogeny, we put $\deg\phi=0$. 
Then $\deg$ is a homogeneous polynomial function $\deg:\End^{\circ}(A) \to \QQ$ of degree
$2\dim A$~\cite[IV.19.Theorem 2]{Mum}, i.e., for any $v_0, v_1\in \End^{\circ}(A)$ the function
$\deg(x_0v_0+x_1v_1)$ is a polynomial in $x_0$ and $x_1$ of degree
$2\dim A$.

An important example of an isogeny is \emph{the arithmetic Frobenius morphism} $F_A: A\to A$. By definition, $F_A$ is trivial on schematic points and raises functions to their $q$-th powers. It is known that $F_A$ is an isogeny of degree $q^{\dim A}$~\cite[Page 205]{Mum}.

Let $A[m]$ be the group subscheme of $A$ annihilated by $m$. 
Fix a prime number $\ell\neq p$. The Frobenius morphism induces an action on the group 
$A[\ell^r](\kk)$, on the Tate module \[T_\ell(A)=\varprojlim_r A[\ell^r](\kk),\] and on
$V_\ell(A)=T_\ell(A)\otimes_{\ZZ_\ell}\QQ_\ell$. 
The following theorem goes back to Weil.

\begin{thm}\label{weil2}
The module $T_\ell(A)$ is a free $\ZZ_\ell$-module of rank $2\dim A$.
The arithmetic Frobenius morphism $\Fr_A$ induces on $V_\ell(A)$ a semisimple linear operator $\Fr$. Its characteristic polynomial $f_A(t)=\det(t-\Fr)$ is a monic polynomial of degree $2\dim A$ with integer coefficients. It does not depend on the choice of $\ell$. 
For any $n\in\ZZ$ we have $f_A(n)=\deg(n-\Fr)$. 
If $\pi$ is a root of $f_A$, then $f_A(q/\pi)=0$, and $|\pi|=\sqrt{q}$.  \qed
\end{thm}
\begin{proof}
  Combine~\cite[IV.19. Theorem 4]{Mum}, and~\cite[IV.21. Theorem 4]{Mum}.
\end{proof}

The polynomial $f_A$ is called \emph{the Weil polynomial of $A$}.
The key result of the Tate--Honda theory states that
abelian varieties $A$ and $B$ over $k$ are isogenous if and only if $f_A=f_B$~\cite{Ta66}.
If $A$ is simple, then $f_A$ is irreducible; in particular, $f_A$ is uniquely determined by any root $\pi_A$. Thus, the isogeny class of a simple abelian variety is determined by $\pi_A$.

The polynomial $f_A$ uniquely determines the endomorphism algebra of the abelian variety.
For example, the endomorphism algebra of a non-supersingular elliptic curve $B$
is the imaginary quadratic extension $\QQ[t]/f_B(t)\QQ[t]$ of $\QQ$~\cite{Wa}.
If a simple abelian surface $A$ is ordinary or mixed, then its endomorphism algebra
is a CM-field and is isomorphic to $\QQ(\pi)=\QQ[t]/f_A(t)\QQ[t]$~\cite{Ta66}.

The roots of $f_A$ are called the {\it Weil numbers}. If $\pi$ is a Weil number, then $q/\pi$ is a Weil number by Theorem~\ref{weil2}. Therefore, \[f_A(t)=\prod_{i=1}^{\dim A}(t-\pi_i)(t-q/\pi_i)\] for some Weil numbers $\pi_1,\dots, \pi_{\dim A}$. 
Note that $\pi+q/\pi$ is a real algebraic number. We say that \[h_A(t)=\prod_{i=1}^{\dim A}(t-\pi_i-q/\pi_i)\] is \emph{the real Weil polynomial of $A$}. 
                                                              
We need the following lemma, which allows us to construct
new abelian varieties within a given isogeny class.

\begin{lemma}~\cite[IV.2.3]{Milne}\label{k_iso_lemma}
Let $\phi:B\to A$ be an isogeny. Then $T_\ell(\phi):T_\ell(B)\to T_\ell(A)$
is a $\ZZ_\ell$--linear embedding commuting with the Frobenius action,
whose image generates $V_\ell(A)$.

Conversely, for each $\ZZ_\ell$--submodule $T\subset T_\ell(A)$, invariant
under the Frobenius morphism and generating $V_\ell(A)$, there exist
an abelian variety $B$ and an isogeny $\phi:B\to A$ over $k$ such that
$T_\ell(\phi):T_\ell(B)\to T$ is an isomorphism.
\end{lemma}

\subsection{Finite subschemes of abelian varieties}\label{fin_subsch}
In this section we recall some results on finite group schemes from~~\cite{How96}.
We fix an isogeny class $\C$ of abelian varieties of dimension $g$.
We have a corresponding endomorphism algebra $\End^\circ(\C)$ and an integer element $F\in\End^\circ(\C)$
such that for every $A$ in $\C$ there is an isomorphism $i_A:\End^\circ(\C)\cong\End^{\circ}(A)$ such that $i_A(F)=F_A$. 
Consider a subring $R=\ZZ[F,V]$ in $\End^{\circ}(\C)$  generated over $\ZZ$ by $F$
 and a shift (Verschiebung) $V=q/F$. 
 The algebra $K=R\otimes\QQ$ is a product of real or CM-fields;
 therefore, the complex conjugation defines an involution $r\mapsto\bar r$ on $K$.
Moreover, $\Bar F=V$; therefore, $R$ is invariant under this involution.

If $\ell\neq p$, the group scheme $A[\ell]$ is reduced and is uniquely determined by the Frobenius action on the vector space $A[\ell](\kk)$ of dimension $2\dim A$ over $\FF_\ell$. 
By Theorem~\ref{weil2}, the characteristic polynomial
of this action on $A[\ell](\kk)$ is equal to $f_A(t)\bmod\ell$.
Note that the Frobenius action is not always semi-simple on $A[\ell](\kk)$, and
 $A[\ell]$ is not always uniquely determined by the polynomial $f_A$.
This issue can be avoided using Grothendieck groups.

Let $\IK$ be the full subcategory of the category of finite commutative
group schemes $\Delta$ such that there exists a monomorphism from $\Delta$ to some abelian variety from $\C$. 
In other words, $\IK$ is the category of kernels of isogenies between varieties in $\C$.
  The category $\IK$ is a product of four subcategories $\IK_{rr}$,
$\IK_{rl}$, $\IK_{lr}$, and $\IK_{ll}$, where the first index means that
objects of these categories are respectively reduced or local, and
the second means the same for the Cartier dual of each object.

For a finite $\ZZ$-algebra $S$ we denote by $\Mod_S$ the category of finite $S$-modules.
There is a natural functor $\P$ from $\IK_{rr}\oplus\IK_{rl}$ to $\Mod_R$ given by $X\mapsto X(\kk)$, where the $R$-module structure on $X(\kk)$ is given by the action of Frobenius and Verschiebung on $X$.
The functor $\P$ can be extended to $\IK_{lr}$ as follows.
If $M$ is a finite $R$-module, then \[\widehat M= \Hom_\ZZ(M, \QQ/\ZZ)\] is naturally an $R$-module:
 for every $r\in R$ and $\psi\in \widehat M$ we set \[(r\psi)(m)=\psi(\bar r m)\] for all $m\in M$.
For a local group scheme $X$ with reduced Cartier dual scheme $D(X)$ we define $\P(X)$ by the formula: \[\P(X)=\widehat{D(X)(\kk)}. \]

The Grothendieck group $G(\IK)$ of $\IK$ is defined as the quotient of the free abelian group
generated by isomorphism classes of objects in $\IK$ by the subgroup generated by the expressions $Y-X-Z$ for any exact sequence
\[0\to X\to Y\to Z\to 0\] in $\IK$. Denote by $[X]$ an element of $G(\IK)$ corresponding to $X$. 
For an $R$-module $M$ we denote by $[M]_R$ the corresponding element of the Grothendieck group $G(\Mod_R)$ of the category of
finite $R$-modules.

There is a natural homomorphism of Grothendieck groups \[\eps:G(\IK)\to G(\Mod_R).\]
On $\IK_{rr}\oplus\IK_{rl}\oplus \IK_{lr}$ the morphism $\eps$ is induced by $\P$. 
In the remaining case, when $X$ and its Cartier dual $D(X)$ are local, we define $\eps$ as follows.
The category $K_{ll}$ (if it is not empty) is generated by a single simple object
$\alpha_p$. This is the group scheme $\Spec k[t]/(t^p)$ with co-multiplication
$t\mapsto t\otimes 1 + 1\otimes t$. 
Therefore, it is enough to define $\eps([\alpha_p])$.
The corresponding $R$-module is the abelian group $\ZZ/p\ZZ$ equipped with the trivial action of $F$ and $V$.

\begin{thm}~\cite[Theorem 3.1]{How96}
The morphism $\eps:G(\IK)\to G(\Mod_R)$ is a well-defined isomorphism.\qed
\end{thm}

Let $\alpha=a/b\in K=R\otimes\QQ$ be an invertible element. {\it The principal element} corresponding to $\alpha$ is
\[\Pr_R(\alpha)=[R/aR]_R-[R/bR]_R\in G(\Mod_R).\]

\begin{thm}~\cite[Theorem 3.5]{How96}
If $\alpha\in K$ is an isogeny, then \[\eps([\ker\alpha])=\Pr_R(\alpha).\qed\]
\end{thm}


\subsection{Kernels of polarizations on abelian varieties over finite fields.}
This section is based on the main results of the paper~\cite{How96} about
principal polarizations on abelian varieties over finite fields.

Let $R^+=\ZZ[F+V]\subset R$, and let $K^+=R^+\otimes\QQ\subset K=R\otimes\QQ$.
Denote by $\OO$ the ring of integers of $K$, and let $\OO^+=K^+\cap\OO$ be the ring of integers of $K^+$. 
Choose an order $R\subset S$ in $K$ stable under involution, and put $S^+=S\cap \OO^+$.

Define an involution $P\mapsto \bar P$ on $G(Mod_S)$ by the formula:
 \[\overline{[M]}_S=[\widehat M]_S.\]
Since any $S$-module is naturally an $S^+$-module, the formula $N_{S/S^+}([M]_S)=[M]_{S^+}$ defines
a norm $N_{S/S^+}:G(Mod_S)\to G(Mod_{S^+})$.
Let $Z(S)$ be the subgroup of symmetric elements in the kernel
of \[G(Mod_S)\to G(Mod_{S^+})\otimes\ZZ/2\ZZ,\] induced by the norm. 
It is known that $\Pr(Tp(K))\subset Z(R)$~\cite[Section 5]{How96}.

An ample line bundle $L$ on $A$ is called a polarization. It defines an isogeny to the dual abelian variety \[\phi_L:A\to \check{A}.\]
The kernel of this isogeny $\ker L$ is called \emph{the kernel of polarization}. 
A polarization is \emph{principal} if the kernel is trivial.
The group scheme $\ker L$ is endowed with the Weil pairing $e_L:\ker L\times \ker L\to G_m$~\cite[\S 23]{Mum}.

A finite group scheme $X$ from $\IK$ is called
{\it attainable in the isogeny class $\C$} if there exists an abelian variety $A$ in $\C$ 
and a polarization $L$ on $A$ such that $X=\ker L$. 
An element $P$ of $G(\Mod_R)$ is called \emph{attainable} if it
is \emph{effective}, that is, $P=[M]_R$ for some $R$-module $M$, and
there exists a group scheme $X$ attainable in $\C$ such that
$\eps([X]_\C)=P$. 
We say that an ideal $\p\subset R$ is \emph{attainable}, if $[R/\p]_R$ is attainable.

 Let $B(S)=\{P+\bar P: P\in G(\Mod_S)\}$, and let $\B(S)$ be $Z(S)/(\Pr(Tp(K))B(S))$.

\begin{thm}~\cite[Theorem 1.3]{How96}
\label{Howe_thm}
Let $\C$ be an isogeny class of abelian varieties over $k$.
Then there exists an element $I\in\B(R)$ such that attainable elements in $G(\Mod_R)$ 
are exactly the effective elements of $Z(R)$ that belong to the class $I$.

In particular, $\C$ contains a principally polarized abelian variety if and only if $I=0$.\qed
\end{thm}

We are going to study attainable primes.
The group $Z(R)/B(R)$ is a vector space over $\FF_2$, with basis formed by the classes of the form $[R/\p]_R$,  where $\p$ is the maximal ideal of $R$ stable with respect to involution and such that the degree extension of $R/\p$ over $R^+/(R^+\cap\p)$ is two. 
We call such ideals $\p$ \emph{generating}, and denote the ideal $R^+\cap\p$ by $\p^+$.

Let $M$ be a finite $R$-module. We denote by $M_\p$ the localization at a prime
ideal $\p\subset R$. 
It follows from the Chinese Remainder Theorem~\cite[Theorem 2.13]{Eisenbud} that
\[M\cong\oplus_\p M_\p,\] 
where the sum is taken over all maximal ideals $R$. 
The set of prime ideals $\p$ such that $M_\p\neq 0$ is called \emph{the support} of $M$.

\begin{rem}\label{loc_eq}
Since the localization functor is exact, if $M$ and $N$ are equivalent in $G(\Mod_R)$, then $M_\p$ and $N_\p$ are also equivalent.
Let $[M]_R=\eps([\Delta])$, where $\Delta$ is a finite group scheme.
Then $[M_\p]_R=\eps([\Delta_\p])$, where $\Delta=\oplus_\p \Delta_\p$. 
\end{rem}

We now prove that if $\p$ is not a generating ideal, because $R/\p\cong R^+/\p^+$, then it is not attainable.

\begin{lemma}\label{key_rem}
Let $\Delta$ be a kernel of polarization on $A$, and let $\p$ be a symmetric ideal such that $R/\p\cong R^+/\p^+$. 
Then $\epsilon([\Delta_\p])\neq [R/\p]_R$.
\end{lemma}
\begin{proof}
If the symmetric ideal $\p$ has the property $R/\p\cong R^+/\p^+$, then the class $[R/\p]_R$ does not belong to $Z(R)$.
Therefore, we have to prove that $\epsilon([\Delta_\p])$ belongs to $Z(R)$.

Let $M$ be an $R$-module such that $[M]_R= \epsilon([\Delta])$; in particular, $\epsilon([\Delta_\p])=[M_\p]_R$.
 According to Theorem~\ref{Howe_thm}, $N_{R/R^+}(M)$ is equivalent to $N\oplus N$ for some $R^+$-module $N$.
Since $\p$ is symmetric, $N_{R/R^+}(M_\p)\cong N_{R/R^+}(M)_{\p^+}$.
By Remark~\ref{loc_eq},  $N_{R/R^+}(M)_{\p^+}$ is equivalent to $N_{\p^+}\oplus N_{\p^+}$;
therefore, $[M_\p]_R\in Z(R)$.
\end{proof}

We say that an isogeny class $\C$ is \emph{exact}, if the following sequence is exact
\[Z(\OO)/B(\OO)\stackrel{N}{\rightarrow} Z(R)/B(R)\to\B(R)\to 0, \eqno(*)\]
where $N$ is induced by the norm $N_{\OO/R}:G(\Mod_\OO)\to G(\Mod_R)$.
In particular, an element $x\in Z(R)/B(R)$ vanishes in $\B(R)$ if and only if it lies in the image of $N$. 

\begin{rem}\label{rem2}
If the isogeny class is exact, and $R$ is maximal at a generating ideal $\p$, then $\p$ is attainable.
\end{rem}

\begin{thm}\label{exact}
An isogeny class $\C$ is exact if and only if there exists a principally polarized abelian variety in $\C$.

An isogeny class $\C$ is exact in the following cases:
\begin{enumerate}
	\item $\dim\C$ is odd;
	\item  $\C$ is the class of a geometrically simple abelian surface. \qed
\end{enumerate}
\end{thm}
\begin{proof}
According to~\cite[Proposition 6.4]{How95}, there is a push-out diagram:
\[\xymatrix{
Z(\OO)/B(\OO) \ar[r]\ar[d]^{N} & \B(\OO) \ar[d]^{i^*}  \\
Z(R)/B(R) \ar[r] & \B(R) } \]

It follows that the isogeny class $\C$ is exact if and only if $i^*$ is the zero map.
According to~\cite[Proposition 7.1]{How95}, and Theorem~\ref{Howe_thm}, if $i^*$ is the zero map, then there exists a principally polarized abelian variety in $\C$. On the other hand, by~\cite[Proposition 6.2]{How95}, $\B(\OO)$ is either trivial or isomorphic to $\ZZ/2\ZZ$. Therefore, if the obstruction element $I$ is trivial, then $i^*$ is the zero map.

The second part of the theorem follows from~\cite[Theorem 1.2]{How95}, and~\cite[Theorem 4.3]{MN}.
\end{proof}


\subsection{Exceptional primes}\label{ExPrimes}
In this section, we examine the case where a generating prime is not attainable.

Let $\p\subset R$ be a maximal ideal over $\ell\in\ZZ$. Let $L$ be an unramified extension of $\QQ_\ell$
with the residue field $R/\p$. Denote by $\Lambda_\p$ the ring of integers of $L$. 
According to~\cite[Proposition 2.5]{Ry14}, the localization $R_\ell\cong R\otimes_\ZZ{\ZZ_\ell}$ is a $\Lambda_\p$--algebra.
In the same way we can start with  $\p^+=\p\cap R^+$, and define $\Lambda^+_\p$ as a ring of integers in an unramified extension $L^+$ of $\QQ_\ell$ with residue field $R^+/\p^+$.
Note that if $\p$ is generating, then $[L:L^+]=2$.
In applications we need only the case $\Lambda^+_\p=\ZZ_\ell$. 

We say that $\ell$ is \emph{an exceptional prime} if there exists a prime ideal $\p_1\subset\OO$ over $\ell$ such that
\begin{enumerate}
\item $\p=\p_1\cap R$ is a generating ideal;
\item if $\p_1^+=\p_1\cap\OO^+$, then  $\dim(\OO^+/\p_1^+)>\dim(R^+/\p^+)$;
\end{enumerate}
In this situation, we say that $\p_1$ is \emph{an exceptional prime ideal}.

\begin{thm}\label{ExPrime}
Let $\C$ be an exact isogeny class, and let $\p_1\subset\OO$ be a maximal ideal such that $\p=\p_1\cap R$ is generating. 
\begin{enumerate}
\item If $\dim(\OO^+/\p_1^+\cap\OO^+)=\dim(R^+/\p^+),$ then the class $[R/\p]_R$ is attainable. 
\item If $\p_1$ is an exceptional prime ideal, then $\p_2=\overline{\p_1}$ is an exceptional prime ideal such that $\p=\p_2\cap R$.
Moreover, $\p_1\neq\p_2$, and $\p$ is \emph{not} attainable.
\end{enumerate}
\end{thm}
 \begin{proof}
Since $\p$ is symmetric, $\p\subset\p_2$. The ideal $\p\subset R$ is maximal; therefore, $\p=\p_2\cap R$. Since $\p$ is generating, the image $\Lambda$ of a natural monomorphism from $\Lambda_\p$ to the localization $K_\ell$ 
is not contained in $R^+$; therefore,
 \[\OO_{\p_1}=\OO^+_{\p_1^+}\cdot\Lambda\cong \OO^+_{\p_1^+}\otimes_{\Lambda^+_\p}\Lambda_\p.\]

Assume first that $\dim(\OO^+/\p_1^+)=\dim(R^+/\p^+)$. Then the maximal unramified subring of $\OO^+_{\p_1^+}$ is equal to the image of $\Lambda^+_\p$.
It follows that \[\OO_{\p_1}\cong\OO^+_{\p_1^+}\otimes_{\Lambda^+_\p}\Lambda_\p\] is a domain and $\p_1=\p_2=\p_1^+\OO$ is symmetric.

Since $\p$ is generating, \[\dim(R/\p)=2\dim(R^+/\p^+).\]
On the other hand, \[\dim(\OO/\p_1)\leq 2\dim(\OO^+/\p_1^+)=2\dim(R^+/\p^+),\] and therefore the embedding
$R/\p\to\OO/\p_1\OO$ is an isomorphism, that is, $[R/\p]_R$ lies in the image of the norm, and $\p$ is attainable.

Assume now that $\dim(\OO^+/\p_1^+)>\dim(R^+/\p^+)$. Then the maximal unramified subring of 
$\OO^+_{\p_1^+}$ is greater than the image of $\Lambda^+_\p$;
since  $[L:L^+]=2$ we have \[\OO^+_{\p_1^+}\otimes_{\Lambda^+_\p}\Lambda_\p\cong \OO^+_{\p_1^+}\oplus\OO^+_{\p_1^+}.\]
In this case we get that $\p_1^+\OO=\p_1\p_2$, and $\p_1\neq\p_2$.

Clearly, $[\OO/\p_1]_\OO\not\in Z(\OO)$, and $[R/\p]_R\not\in B(R)$, because $R/\p$ is a simple $R$-module.
The support of $[R/\p]_R$ is equal to $\p$; therefore, if $[R/\p]_R$ is in the image of the norm, then
it is in the image of the submodule of $Z(\OO)$ generated by $[\OO/\p_1]_\OO+[\OO/\p_2]_\OO$. 
This is nonsense, because the natural morphism $R/\p\to\OO/\p_1$ is injective, and $\dim(R/\p)<\dim(\OO/\p_1)+\dim(\OO/\p_2)$.
The theorem is proved.
\end{proof}

\section{Polarizations on the product of two abelian varieties.}\label{S3}
Recall the Serre construction of the principal polarization on a product of two polarized abelian varieties~\cite{KL}. Let $A_1$ and $A_2$ be abelian varieties with
polarizations $L_1$ and $L_2$ of degree $N^2$. Assume that we are given an
isomorphism $\psi: \Ker L_1\to \Ker L_2$ with the following property:
\[e_{L_2}(\psi(x),\psi(y))=e_{L_1}(x,y)^{-1},\] in other words, $\psi$ is an anti-isometry.
Denote by $\Delta$ the image of the mapping \[Id\times\psi:\Ker L_1\to \Ker L_1\times \Ker L_2.\]
On $J=(A_1\times A_2)/\Delta$ one get a principal polarization as follows.
The kernel of the product polarization on $A_1\times A_2$ is $\Ker L_1\times \Ker L_2$.
The restriction of the product form $e_{L_1}\times e_{L_2}$ to $\Delta$ is trivial. According to~\cite[\S 23,Theorem 2]{Mum}, 
the polarization on $A_1\times A_2$ descends to a principal polarization $M$ on $J$. 
This construction is called {\it the gluing of polarizations}.

We say that the principal polarization $L$ on $J$ is {\it reducible} if
there exist abelian subvarieties $J_1$ and $J_2$ with polarizations
$L_1$ and $L_2$ such that $J\cong J_1\times J_2$ and $L$ is the product of
polarizations $L_1$ and $L_2$. Otherwise, we call the polarization {\it irreducible}.

Let $E$ be the set of integers $e$ such that for any finite group scheme $\Delta$
over $k$ that can be embedded in a variety isogenous to $A_1$ and in a variety isogenous
to $A_2$ we have $e\Delta=0$. By definition, \emph{the gluing exponent} $e(A_1,A_2)$ is the greatest common divisor of the set $E$~\cite{HL}.
Note that $e(A_1,A_2)$ is finite if and only if there is no abelian variety with a nonzero morphism to both $A_1$ and $A_2$.

\begin{lemma}\label{glue}
Let $A_1$ and $A_2$ be simple abelian varieties, and let there exist polarizations $L_1$ on $A_1$ and $L_2$ on $A_2$ of degree $N^2>1$ and an anti-isometry $\Ker L_1\to \Ker L_2$.
If $e(A_1,A_2)$ is finite, then the gluing of $L_1$ and $L_2$ is irreducible.
\end{lemma}
\begin{proof}
Assume that the gluing of the polarizations $\alpha:A_1\times A_2\to J$ is reducible: $J\cong J_1\times J_2$.
Let the composition $\alpha:A_1\times A_2\to J_1\times J_2$ be given by the matrix $\alpha_{ij}:A_i\to J_j$.
Since $A_1$ and $A_2$ are simple and not isogenous, $\alpha$ splits into the product of two
isogenies, say $\alpha_{11}:A_1\to J_1$, and $\alpha_{22}:A_2\to J_2$.
Therefore, the kernel of $\alpha$ is the direct sum of the kernels of $\alpha_{11}$ and $\alpha_{22}$,
but this contradicts the fact that $\alpha$ induces embeddings $A_1\to J$, and $A_2\to J$.
\end{proof}

We will compute $e(A_1,A_2)$ in some cases. For this, we need a lemma.

\begin{lemma}\label{hb_lemma}
Let $f(t)$ be a separable Weil polynomial with a real Weil polynomial $h(t)$, and let $\Lambda$ be a ring.
For any $r\in\Lambda$ there exists a polynomial $\lambda(t)\in\Lambda[t]$ such that 
\[f(t)= \lambda(t)(t^2-rt+q)+ t^{\deg h}h(r).\]
\end{lemma}
\begin{proof}
The lemma follows from the equality:
\[f(t)=\prod_x(t-x)(t-q/x)=\prod_x(t^2-b_xt+q)=\prod_x((t^2-rt+q)+(r-b_x)t)=\] 
\[=\lambda(t)(t^2-rt+q)+t^{\deg h}\prod_x(r-b_x)=\lambda(t)(t^2-rt+q)+t^{\deg h}h(r),\]
 where $b_x=x+q/x$, and $h(t)=\prod_x(t-b_x)$.
\end{proof}

\begin{prop}\label{hb}
Let $A$ be a simple abelian variety with real Weil polynomial $h(t)$, and let $B$ be an elliptic curve with real Weil polynomial 
$h_B(t)=t-b$. Suppose that both $f_A$ and $f_B$ are separable. 
Then there exists a natural number $r$ such that $e(A,B)=|h(b)|/p^r$.
Moreover, if $A$ is ordinary and $B$ is supersingular, then $e(A,B)=|h(b)|$. 
\end{prop}
\begin{proof}
There exists a polynomial $a(t)\in\ZZ[t]$ such that \[h(t)=a(t)(t-b)+h(b).\]
Therefore, the reduced resultant of  $h(t)$ and $h_B(t)$ is equal to $h(b)$~\cite[Lemma 2.7]{HL}.
According to~\cite[Proposition 2.8]{HL}, $e(A,B)$ divides $h(b)$. Hence, it is enough to prove that for every prime $\ell\neq p$, if 
$\ell^n$ divides $h(b)$, then there exists an $\ell$-primary group scheme $\Delta$ such that $\ell^{n-1}\Delta\neq 0$, and there are
homomorphisms $\Delta\to A'$ and $\Delta\to B'$, where $A'$ is isogenous to $A$ and $B'$ is isogenous to $B$.

By the Lemma~\ref{k_iso_lemma}, there exists an elliptic curve $B'$ such that
\[T_\ell(B')\cong \ZZ_\ell[t]/f_B(t)\ZZ_\ell[t],\] where $F_B$ acts on the right as multiplication by $t$.
  Put $\Delta=B'[\ell^n]$. 
To prove the proposition we now have to construct a submodule
$T\subset V_\ell(A)$ such that $\P(\Delta)$ is a submodule of $T/\ell^n T$, and $T$ generates $V_\ell(A)$; 
then, according to Lemma~\ref{k_iso_lemma}, there exists an abelian variety $A'$ such that $T_\ell(A')\cong T$,
 and $\Delta$ is isomorphic to a submodule of $A'[\ell^n]$.

Define a sequence of polynomials \[f_r(t)=t^r-\alpha_rt+\beta_r\in\ZZ[t]\] 
for $r\geq 2$ as follows: $f_2=f_B$, and \[f_{r+1}(t)=tf_r(t)+\alpha_rf_B(t).\]
 Clearly, $f_r(F_B)=0$ for all $r\geq 2$. 
Since $f_A$ is separable, there exists a cyclic vector $v\in V_\ell(A)$, i.e., $v, Fv,\dots, F^{d-1}v$ is a basis, where $d=\deg f_A$.
Let $v_0=v, v_1=Fv$, and \[v_r=\frac{1}{\ell^n}f_r(F)v\] for $d>r\geq 2$. 
According to Lemma~\ref{hb_lemma} we have: \[f_A(t)\equiv \lambda(t)f_B(t)\bmod\ell^n\] for some $\lambda(t)\in\ZZ[t]$; 
using this relation it is straightforward to check that $v_0,\dots, v_{d-1}$ is a basis of an $F$-invariant submodule $T\subset V_\ell(A)$. 
Moreover, by induction we get that \[f_r(F)v\equiv 0\bmod\ell^n.\] In other words, $v$ and $Fv$ generate a submodule of $T/\ell^n T$ isomorphic to $\P(\Delta)$. 

If $A$ is ordinary, then $h(0)$ is a $p$-adic unit. Moreover, if $B$ is supersingular, then $p$ divides $b$; therefore,
$h(b)\equiv h(0)\bmod p$, and $v_p(h(b))=v_p(e(A,B))=0$. It follows that $e(A,B)=|h(b)|$. 
\end{proof}

There is a simple way to prove that polarizations on two abelian varieties $A$ and $B$ cannot be glued together.
Let $\ell$ be a prime divisor of $e(A,B)$. If $B$ is supersingular, and all the roots of $f_B$ are real, 
then $B[\ell]$ is semi-simple. If, on the other hand, any kernel of a polarization on $A$ annihilated by $e(A,B)$ is not
semi-simple, then there is no gluing of polarizations on $A$ and $B$. The following proposition allows one to prove 
that the kernel of a polarization is not semi-simple.

\begin{prop}\label{main_prop}
Let $A$ be an abelian variety with polarization $L$, and let $\Delta$ be the kernel of $L$.
Assume that the Weil polynomial $f_A$ of $A$ is separable.
Let $\p$ be a symmetric maximal ideal of $R_A$ such that
\begin{enumerate}
\item $p\not\in\p$;
\item $R_A$ is maximal at $\p$;
\item $\p^+$ is ramified in the extension $K/K^+$.
\end{enumerate}
If the support of $\Delta$ contains $\p$, then $\P(\Delta_\p)\cong R_A/\p^a$,
 where $a=\length(\Delta_\p)>1$.
In particular, $\Delta_\p$ is not semi-simple and 
$(\p^+)^{\left\lceil a/2\right\rceil}\Delta_\p=0$.
\end{prop}
  \begin{proof}
According to Lemma~\ref{NonSS}, $\P(\Delta_\p)\cong R_A/\p^a$, where $a=\length(\Delta_\p)$. 
Since $\p^2=\p^+\OO$, we have 
\[(\p^+)^{\left\lceil a/2\right\rceil}\Delta_\p\subset\p^a\Delta_\p=0.\]

Since	$\p^+$ is ramified, $R/\p\cong R^+/\p^+$.
It follows from Lemma~\ref{key_rem} that $\P(\Delta_\p)\not\cong R_A/\p$.
Therefore, if the support of $\Delta$ contains $\p$, then the multiplicity of $[R_A/\p]_{R_A}$ in $\eps([\Delta])$ is at least two, that is, $a>1$.
The proposition follows.
\end{proof}

\begin{lemma}\label{NonSS}
Let $A$ be an abelian variety with a finite subgroup scheme $\Delta$.
Assume that the Weil polynomial $f_A$ of $A$ is separable.
Let $\p$ be a maximal ideal of $R_A$ such that $p\not\in\p$ and the order $R_A$ is maximal at $\p$.
If the support of $\Delta$ contains $\p$, then $\P(\Delta_\p)\cong R_A/\p^a$,
where $a=\length(\Delta_\p)$.
\end{lemma}
\begin{proof}
The group scheme $\Delta_\ell$ corresponds to a quotient of the Tate module $M=T_\ell(A)/T_M$.
Since the Weil polynomial $f_A$ of $A$ is square-free, the module $V_\ell(A)$ is free of rank $1$ over $R_A\otimes\QQ$.
Since the order of $R_A$ is maximal at $\p$, we find that $T_\ell(A)_\p$ is free of rank $1$ over $R_\p=(R_A)_\p$.
Therefore, the module $M_\p$ is isomorphic to the quotient of $R_\p$ by some power of the maximal ideal: $M_\p\cong R_\p/\p^a$.
\end{proof}

As an application, we reprove a result E.Howe and K.Lauter~\cite{HL}. 
We use the following particular case of the Dedekind criterion.

\begin{thm}[Dedekind]\label{Dedekind}
Let $f\in\ZZ_\ell[t]$ be a polynomial such that \[f(t)= g(t)^r+\ell\alpha(t),\] 
where $g(t)\in\ZZ_\ell[t]$ is irreducible modulo $\ell$, and $\alpha\in\ZZ_\ell[t]$.
Then the ring $\Lambda=\ZZ_\ell[t]/f(t)\ZZ_\ell[t]$ is regular if and only if either $r=1$, or $r>1$ and the polynomials $\alpha(t)$ and $g(t)$ are coprime modulo $\ell$.
\end{thm}
\begin{proof}
    Follows from~\cite[Theorem 6.1.2]{Cohen}.
\end{proof}

\begin{thm}\label{HL}\cite[Theorem 3.1]{HL}
Suppose that $q$ is a square.
Let $A$ be an ordinary abelian variety with a real Weil polynomial $h(t)$, and let $B$
be a power of a supersingular elliptic curve with a real Weil polynomial $(t-2s)^n$, where
$s^2=q$, and $n>0$. 
If the number $h(2s)$ is squarefree, then there is no variety 
in the isogeny class of $A\times B$ with irreducible principal polarization.
In particular, there is no curve over $k$ with the real Weil polynomial equal to $h(t)(t-2s)^n$.
\end{thm}
\begin{proof}
Assume that there is an irreducible principal polarization on a variety $J$ in the isogeny class of $A\times B$.
According to~\cite[Lemma 2.3]{HL} there exist abelian varieties $A'$ and $B'$ in the 
isogeny class of $A$ and $B$ respectively, and a non-trivial group scheme $\Delta$ with monomorphisms to $A'$ and $B'$.
Moreover, $\Delta$ is isomorphic to the kernel of a polarization on $A'$.

Since $\Delta$ is a subgroup scheme of $B'$, the Frobenius action on $\Delta$ is constant.
 We are going to apply Proposition~\ref{main_prop} to $A'$ and show that the Frobenius action on
$\Delta$ is non-constant. This contradiction proves the theorem.

Let $\ell$ be a prime divisor of $h(2s)$. 
Since $A$ is ordinary, $\ell\neq p$, and since $h(2s)$ is squarefree, $f_A$ is separable.
Denote by $\p\subset R_A$ the ideal generated by $F-s, V-s$, and  $\ell$. Clearly, $R_A/\p\cong\FF_\ell$; 
therefore, $\p$ is prime and symmetric. The localization $R_\p$ of $R_A$ at $\p$ is isomorphic to the 
quotient $\ZZ_\ell[t]/f_\p$, where \[f_\p-(t-s)^{2r}=\ell\alpha(t)\] for some $\alpha(t)\in\ZZ_\ell[t]$. 
According to the Dedekind criterion, $R_\p$ is maximal if
and only if $(t-s)$ is coprime to $\alpha(t)$ modulo $\ell$. This is equivalent to 
\[\alpha(s)\not\equiv 0\bmod\ell\text{, and }f_\p(s)\not\equiv 0\bmod\ell^2.\] 
Apply Lemma~\ref{hb_lemma} with $r=2s$. We get that $f_A(s)=s^{\dim A}h(2s)$.
Since $h(2s)$ is squarefree, $f_\p(s)\not\equiv 0\bmod\ell^2$, and $R_\p$ is maximal.
It is now clear, that $\p^+$ is ramified. According to Proposition~\ref{main_prop}, 
the Frobenius action on $\Delta(\kk)$ is non-constant. The theorem is proved.
 \end{proof}

\begin{rem}
The separability of $f_B$ is crucial for Proposition~\ref{hb}. 
In fact, from the proof of Theorem~\ref{HL} it follows that $e(A,B)\neq |h(b)|$, where $b=2s$.
One can use~\cite{Ry14} and prove that if $f_A$ is separable, and $f_B=(t-s)^2$, where $s^2=q$, then for any prime $\ell\neq p$ we have:
$\ell^n$ divides $e(A,B)$ if and only if $\ell^{2n}$ divides $f(s)$, and $\ell^n$ divides $f'(s)$, where $f'$ is the derivative of $f$.
\end{rem}

We now apply the same technique to a product of a supersingular and a non-supersingular surface.
Let $\ell\neq 2,p$ be a prime, and let $s=\pm\sqrt{q}$.
The ring $\Lambda_\ell=\ZZ_\ell[t]/(t^2-q)\ZZ_\ell[t]$ is either $\ZZ_\ell^2$ or 
the ring of integers in an unramified extension of $\QQ_\ell$. 
We say that $\ell$ divides $h(2s)$ if $\ell$ divides $h(2s)$ in $\Lambda_\ell$.

\begin{thm}\label{HL2}
Suppose that $q$ is not a square.
Let $A$ be an ordinary abelian surface with the real Weil polynomial $h(t)=t^2+a_1t+a_2-2q$, and let $B$
be a supersingular simple abelian surface with the Weil polynomial $(t^2-q)^2$.
Assume that if $\ell$ is any prime divisor of $h(2s)$, then $\ell$ is even, and $\ell^2$ does not divide $h(2s)$.
Then there is no variety in the isogeny class of $A\times B$ with irreducible principal polarization.
In particular, there is no curve over $k$ with the Weil polynomial equal to $f_A(t)(t^2-q)^2$.
\end{thm}
\begin{proof}
Assume that there exists a variety in the isogeny class of $A\times B$ with irreducible principal polarization.
As in the proof of Theorem~\ref{HL}, there exist $A'$ and $B'$ in the isogeny class of $A$ and $B$, respectively, and 
a kernel $\Delta$ of a polarization on $A'$ such that there exists a monomorphism $\Delta\to B'$.
Clearly, $\Delta$ is semi-simple, because, the polynomial $t^2-q$ is separable modulo $\ell$.
 We now deduce from Proposition~\ref{main_prop} that $\Delta$ is not semi-simple. 

Let $\ell$ be an odd prime divisor of $h(2s)$. 
Since $A$ is ordinary, $\ell\neq p$, and since $h(2s)$ is squarefree, $f_A$ is separable.
Let $\pi_1,\pi_2,q/\pi_1,q/\pi_2$ be the roots of $f_A$, and let $b_i=\pi_i+q/\pi_i$, where $i=1$, or $i=2$.
Since $\ell$ divides $h(2s)$ we have \[b_1b_2=a_2-2q\equiv -4q\bmod\ell,\] and $\ell$ divides $b_1+b_2=-a_1$. It follows that
\[f(t)=(t^2-b_1t+q)(t^2-b_2t+q)=(t^2+q)^2+a_1t(t^2+q)+b_1b_2t^2\equiv (t^2-q)^2\bmod\ell.\]

Assume first that the polynomial $t^2-q$ is irreducible modulo $\ell$.
In this case $K^+$ is inert in $\ell$, and $\p^+=\ell\OO^+$ is ramified in the extension $K/K^+$, 
i.e., $\ell\OO=\p^2$. We are going to prove that the order $R_A$ is maximal at $\ell$. 
According to the Dedekind criterion, we have to show that $(t^2-q)$ is coprime modulo $\ell$ to
\[\frac{1}{\ell}(f_A(t)-(t^2-q)^2)=\frac{t}{\ell}(a_1t^2+(a_2+2q)t+qa_1).\] 
This easily follows from the assumption that $\ell^2$ does not divide $h(2s)$.

Assume now that $s\in\ZZ_\ell$, and $t^2-q\equiv(t-s)(t+s)\mod\ell$.
Then $f(t)=f_+(t)f_-(t)$ is a product of monic polynomials such that $f_\pm(t)\equiv (t\pm s)^2\bmod\ell$.
According to the Dedekind criterion, $R_A$ is maximal if and only if $f_\pm(s)\not\equiv 0\bmod\ell^2$; and the last assertion is true, because $\ell^2$ does not divide $h(2s)$.
\end{proof}

\section{Principal polarizations on abelian threefolds.}\label{S4}
In this section we assume that $A$ is a geometrically simple abelian surface; since supersingular abelian 
surfaces are not geometrically simple, $A$ is either an ordinary surface or a mixed one, so $K=\End^\circ(A)$ is a CM-field. 
According to Proposition~\ref{exact}, the isogeny class of $A$ is exact, and 
we can assume that there is a principal polarization $L$ on $A$. 
Let $B$ be an elliptic curve with commutative endomorphism algebra $K_B=\End^\circ(B)$.
 Denote by $\Delta_B$ the discriminant of $K_B$.
Let $R=R_A$ and $R_B$ be the orders generated by the Frobenius and Verschiebung endomorphisms in
$\End^\circ(A)$ and $\End^\circ(B)$ respectively.
It is natural to claim that if there exists a divisor $\ell$ of $e(A,B)$, then there is a gluing of polarizations on $A$ and $B$.
Thanks to Proposition~\ref{hb}, we can use $h(b)$ instead of $e(A,B)$. 
Firstly, we examine the case of an exceptional divisor $\ell$ of $h(b)$.

\begin{lemma}\label{lemma1}
\begin{enumerate}
\item A prime $\ell$ is exceptional if and only if
\begin{itemize}
\item $f_A(t)\equiv f(t)^2\bmod\ell^2$, where $f\in\ZZ_\ell[t]$ is irreducible modulo $\ell$;
\item $\ell$ is inert in $K^+$.
\end{itemize}
\item If $f_A(t)\equiv f(t)^2\bmod\ell$, and $f(t)$ is irreducible modulo $\ell$, then 
$[A[\ell]]_R=2[X]_R$, where $X(\kk)$ is a two-dimensional vector space over $\FF_\ell$,
 and $F$ acts on  $X(\kk)$ with characteristic polynomial $f$.
Moreover, either $[X]_R$ is attainable, or $\ell$ is exceptional.
\item If $\ell$ is exceptional, then $\ell^2$ divides $a_1^2-4a_2+8q$.
\end{enumerate}
\end{lemma}
\begin{proof}
It is straightforward to check that if \[f_A(t)\equiv f(t)^2\bmod\ell,\] then $\ell$ divides the discriminant of the real Weil polynomial $a_1^2-4a_2+8q$; therefore, $\dim(R^+/\p^+)=1$.
If $\ell$ is split or ramified in $K^+$, then Theorem~\ref{ExPrime}.(1) shows that $[X]_R$ is attainable. 
If $\ell$ is inert in $K^+$, then \[\dim(\OO^+/\p_1^+)=2>\dim(R^+/\p^+),\] i.e., $\ell$ is exceptional. This proves $(2)$, and the ``if'' part of $(1)$.

If $\ell$ is exceptional, then, by definition, \[\dim(\OO^+/\p_1^+)>\dim(R^+/\p^+),\] i.e., $R^+_\ell\neq\OO^+_\ell$; 
therefore, $\ell$ divides the discriminant of the real Weil polynomial $a_1^2-4a_2+8q$.
On the other hand, the same inequality shows that $\dim(\OO^+/\p_1^+)=2$, i.e., $\ell$ is inert in $K^+$. 
It follows that $\ell^2$ divides $a_1^2-4a_2+8q$,
and there exists $s\in\ZZ_\ell$ such that \[h(t)\equiv (t-s)^2\bmod\ell^2.\] In particular, $a_1\equiv -2s\bmod\ell^2$.
Put $f(t)=t^2-st+q$. We have
\[f_A(t)=(t-\pi_1)(t-q/\pi_1)(t-\pi_2)(t-q/\pi_2)=(t^2-b_1t+q)(t^2-b_2t+q)=\] 
\[=(f(t)+(s-b_1)t)(f(t)+(s-b_2)t)=f(t)^2+tf(t)(2s+a_1)+t^2h(s)\equiv f(t)^2\bmod\ell^2.\] Part $(1)$ is proved.
\end{proof}

\begin{rem}\label{ExRem}
If $\ell$ is exceptional, then, according to Theorem~\ref{ExPrime}, the group scheme $X$ from the previous lemma is not attainable.
Therefore, for any isogeny $A'\to A$ and any polarization $L$ on $A'$ 
\[\ker L\not\cong B[\ell].\] This observation does not give an obstruction to gluing of polarizations.
Indeed, assume that $\ell$ divides $h_A(b)$. According to Lemma~\ref{lemma1}, 
$\ell^2$ divides the discriminant of the real Weil polynomial $a_1^2-4a_2+8q$, and we have
\[4h(b)=4b^2+4a_1b+4a_2-8q\equiv 4b^2+4a_1b+a_1^2\equiv (2b+a_1)^2\bmod 4\ell^2.\]
Therefore, $\ell^2$ divides $h(b)$, and $f_A(t)\equiv f_B(t)^2\bmod \ell^2$.
Moreover, the proposition below shows that $2[X]_R$ is attainable.
\end{rem}

\begin{prop}\label{prop1}
Let $A$ be a geometrically simple ordinary abelian surface with a principal polarization, and let $\ell$ be an exceptional prime. 
Then \[f_A(t)\equiv f(t)^2\bmod\ell^2\] for some $f\in\ZZ_\ell[t]$, and there is an isogeny $A'\to A$ of degree $\ell$, and a polarization $L$ on $A'$ such that
\[\P(\ker L)\cong \Lambda/\ell^2\Lambda,\] where $\Lambda\cong\ZZ_\ell[t]/f\ZZ_\ell[t]$.
\end{prop}
\begin{proof}
There exists a morphism $\sigma:T_\ell(A)\to K_\ell$ such that the Weil pairing on $T_\ell(A)$ is given by
\[e_A(x,y)=\tr_{K_\ell/\QQ_\ell}(\xi'\sigma(x)\overline{\sigma(y)})\] for some $\xi'\in K$ 
such that $\overline{\xi'}=-\xi'$~\cite[Section 6]{How95}. 

Since $\ell$ is exceptional, $R^+_\ell\neq\ell\OO_\ell^+$; 
thus, the image $\Lambda$ of a natural monomorphism from $\Lambda_\p$ to the localization $K_\ell$ 
is not contained in $R^+_\ell$. Moreover, the local rings $\OO^+_\ell$ and $\Lambda$ are isomorphic;
therefore, \[\OO_\ell\cong\Lambda\otimes\Lambda\cong\Lambda\oplus\Lambda\] as $\Lambda$-modules, and
conjugation acts componentwise: \[\overline{(x_1,x_2)}=(\Bar{x}_1,\Bar{x}_2)\in\Lambda\oplus\Lambda.\]
Let $\Lambda\cong\ZZ_\ell[z]$, where $z^2\in\ZZ_\ell$, and $\Bar z=-z$. Then $\xi'/z\in\OO_\ell^+$, and the image of $\xi'$ in 
$\Lambda\oplus\Lambda$ is $(z\xi,z\Bar\xi)$ for some $\xi\in\Lambda$.

It follows that the Weil pairing corresponds to the pairing 
\[e((x_1,x_2),(y_1,y_2))=\tr_{L/\QQ_\ell}(z\xi x_1 \Bar{y}_1)+\tr_{L/\QQ_\ell}(z\Bar{\xi} x_2\Bar{y}_2) \eqno{(*)}\]
on $\Lambda\oplus\Lambda$, where without loss of generality we can assume that $\xi\in\Lambda$ is a unit, 
and $e$ is perfect on $\OO_\ell$.

The induced pairing on the $\Lambda$-submodule $T=\sigma(T_\ell(A))$ of $K_\ell$ is perfect.
Assume that $T\subset\OO_\ell$. We claim that in this case $T=\OO_\ell$. Indeed, 
 if $T\subset\ell\OO_\ell$, then clearly the pairing on $T$ is not perfect; 
therefore, the $\Lambda$-module $M=\OO_\ell/T$ is cyclic.
Assume that $M\neq 0$.
Let $v\in\OO_\ell$ generate $M$; hence, there exists a minimal natural $r>0$
such that $\ell^rv\in T$. Since the pairing on $T$ is perfect, there exists $u\in T\subset\OO_\ell$ such that $e(\ell^rv,u)\in\ZZ^*_\ell$.
It follows that $e(v,u)\not\in\ZZ_\ell$. Nonsense.

The inequality $\dim(\OO^+/\p_1^+)>\dim(R^+/\p^+)$ shows that $R_\ell\neq\OO_\ell$, and thus
there exists a natural $m$ such that \[R_\ell=\Lambda+\ell^m\OO_\ell.\]
It follows that $T'=\Lambda\oplus\ell \Lambda\subset\OO_\ell$ is an $R$-submodule of $T$.
Clearly, the kernel of $e$ on $T'$ is isomorphic to $\Lambda/\ell^2\Lambda$.

Assume now that $T\not\subset\OO_\ell$. 
The  $\Lambda$-module $M=T/(T\cap\OO_\ell)$ is cyclic, because otherwise the pairing on $T$ is not integral.
Let $v\in T$ generate $M$ over $\Lambda$, and let $u\in T\cap\OO_\ell$ be a second basis element 
 such that $e(v,u)\in\ZZ^*_\ell$.
It follows that $T'=\ell\Lambda v\oplus\Lambda u$ is an $R$-submodule of $T$ such that the kernel of the restriction of $e$ to $T'$ is isomorphic to $\Lambda/\ell^2\Lambda$.

According to Lemma~\ref{k_iso_lemma}, there exists an abelian surface $A'$ and an isogeny $A'\to A$ such that the kernel of the induced 
polarization $L$ on $A'$ is isomorphic to $\Lambda/\ell^2\Lambda$. Now, Lemma~\ref{lemma1} completes the proof.
\end{proof}

\begin{rem}
The proof of Proposition~\ref{prop1} is based on the Deligne equivalence theorem used in~\cite[Section 6]{How95}.
I think the proposition has to be true for any simple abelian surface with commutative endomorphism algebra.
\end{rem}

\begin{lemma}\cite[Lemma 4.2]{HNR06}\label{lemma_HNR}
Let $\ell$ be a prime number and $K_1$ be an imaginary quadratic field
whose discriminant is not equal to $-\ell$. Then there are infinitely many
primes $r'$ that split in $K_1$ and are not squares modulo $\ell^n$, where $n>1$, if $\ell=2$.
\end{lemma}

We need the following generalization of~\cite[Lemma 4.3]{HNR06}.

\begin{lemma}\label{lemma2}
Let $B$ be an elliptic curve.
Suppose that there exists a surface $A$ with a polarization $L$, and an isomorphism $\psi:\ker L\to B[\ell^n]$, where $n>1$. If $\Delta_B\neq-\ell$, then there exists a curve $B'$ in the isogeny class of $B$, 
and an anti-isometry $B'[\ell^n]\to\ker L$.
\end{lemma}
\begin{proof}
Let $e_B$ and $e_X$ be Weil pairings on $B[\ell^n]$ and $X=\ker L$, respectively.
Then there exists $r$ such that the following diagram is commutative:
\[
\begin{CD}
X\times X @>>>B[\ell^n]\times B[\ell^n]\\
@VV e_X V @VV e_B V\\
G_m @>> r> G_m\\
\end{CD}
\]

Take a composition of $\psi$ with an isogeny $\phi:B\to B'$ of degree
$r'$, then the number $r$ is multiplied by $r'$. If $-r$ is a square modulo $\ell^n$, then it suffices to multiply $\psi$ by $r'$ such that $r(r')^2\equiv -1\bmod\ell$. 
Then $r'\psi$ is an anti-isometry.

Suppose that $-r$ is not a square modulo $\ell^n$. If $\Delta_B\neq-\ell$, then, by Lemma~\ref{lemma_HNR}, one can find $r'$,
which is not a square modulo $\ell^n$ and splits in $K_1$,
hence $B[r']\cong X_1\oplus X_2$ splits, and there exists an
isogeny $\phi: B\to B'=B/X_1$ of degree $r'$. 
There are infinitely many such $m$ by the Dirichlet theorem on primes in arithmetic
progressions; as before, for such $r'$ the number $-rr'$ is a square modulo $\ell$.
The lemma is proven.
\end{proof}

\begin{proof}[Proof of Theorem~\ref{thm_f}]
Suppose that $\ell\neq p$ divides $h(b)$. According to Lemma~\ref{hb_lemma}, 
$f_B(t)$ divides $f_A(t)$ modulo $\ell^n$ if and only if $\ell^n$ divides $h(b)$.  
We are going to prove that under conditions of the theorem there exists a gluing of polarizations on $A$ and $B$.

If $f_B(t)$ is irreducible modulo $\ell$, then $[A[\ell]]_R=[X]_R+[X']_R$, where $X\cong B[\ell]$.
If $f_A(t)\not\equiv f_B(t)^2\bmod\ell$, then $f_A\equiv f_Bf'\bmod\ell$ is a product of two coprime monic polynomials modulo $\ell$;
therefore, $[X]_R\neq [X']_R$, and, by Hensel's lemma, $f_A=f_1f_2$, 
where $f_1,f_2\in\ZZ_\ell[t]$ are coprime and monic, and $f_1\equiv f_B\bmod\ell$.
According to the Chinese Remainder Theorem, \[R\cong \ZZ_\ell[t]/f_1\ZZ_\ell[t]\oplus\ZZ_\ell[t]/f_2\ZZ_\ell[t],\] 
and, by Theorem~\ref{Dedekind}, the order $R_1=\ZZ_\ell[t]/f_1\ZZ_\ell[t]$ is maximal. By Remark~\ref{rem2}, $[X]_R$ is attainable.
According to Lemma~\ref{lemma1}, if $f_A(t)\equiv f_B(t)^2\bmod\ell$, then $[X]_R$ is either attainable, or
$\ell$ is exceptional. 

Assume that $\ell$ is exceptional. 
By Remark~\ref{ExRem}, $\ell^2$ divides $h(b)$, and $f_A(t)\equiv f_B(t)^2\bmod\ell^2$.
According to Proposition~\ref{prop1}, $A$ is isogenous to an abelian surface $A'$ 
with a polarization $L$ such that $\ker L\cong B[\ell^2]$.
 
Let us now consider the case when the polynomial
$f_B(t)\equiv (t-t_1)(t-t_2)\bmod\ell$ is reducible modulo $\ell$.
Then \[[B[\ell]]_{R_B}=[Y_1]_{R_B}+[Y_2]_{R_B},\] 
where both $Y_1(\kk)\cong Y_2(\kk)\cong\FF_\ell$ and $F$ acts on $Y_i$ as multiplication by $t_i$.
Therefore, \[[A[\ell]]_R=[Y_1]_R+[Y_2]_R+[X']_R.\] 
Clearly, $[Y_1]_R=\overline{[Y_2]_R}$, and from Theorem~\ref{Howe_thm} it follows that $A$ is isogenous to a surface $A'$
with a polarization $L$, such that $[\ker L]_R=[Y_1]_R+[Y_2]_R$.
If $t_1\not\equiv t_2\bmod\ell$, then \[\ker L\cong Y_1\oplus Y_2=B[\ell].\]

Suppose that $t_1\equiv t_2\bmod\ell$. The Frobenius action on $\ker L(\kk)$ is either a multiplication by $t_1$, or is given by a non-trivial Jordan cell with eigenvalue $t_1$. Let $S_\ell$ be the localization of ${R_B}$ at $\ell$.
By assumption, $\ell^2$ divides $f_B(t_1)$, and, according to the Dedekind criterion,
the order $S_\ell$ is not maximal. Let $\OO_\ell$ be the maximal order in $S_\ell\otimes\QQ_\ell$;
then the Frobenius action on $\OO_\ell/\ell\OO_\ell$ is the multiplication by $t_1$, and on $S_\ell/\ell S_\ell$ is given by
a non-trivial Jordan cell with eigenvalue $t_1$.
By Lemma~\ref{k_iso_lemma}, there exist elliptic curves $B_1$ and $B_2$ such that $T_\ell(B_1)\cong S_\ell$, and
$T_\ell(B_2)\cong \OO_\ell$, as ${R_B}$-modules.
 Therefore, either $\ker L\cong B_1[\ell]$, or  $\ker L\cong B_2[\ell]$.

Consider the case when $p$ divides $h(b)$. Suppose that $B$ is ordinary. Then $B[p]\cong X_r\oplus X_l$,
where $X_r\otimes\kk\cong\ZZ/p\ZZ$, and $X_l\otimes\kk\cong\mu_p$,
since $B$ is not supersingular~\cite{Wa}. The Weil pairing is non-degenerate; therefore, we have an isomorphism of $X_r$ with its Cartier dual $D(X_l)$. We get the equality $[X_r]_{R_B}=\overline{[X_l]_{R_B}}$. 
By Theorem~\ref{Howe_thm}, the element $[X_r]_R+[X_l]_R$ is an attainable group subscheme of $A$.
Suppose now, that $B$ is supersingular. Since $p$ divides both $b$ and $h(b)$, the surface $A$ is mixed. 
Thus, there exists a monomorphism $B[p]\to A$. Moreover, $[B[p]]=2[\alpha_p]$ is an attainable group subscheme of $A$.

We have proved that there is an abelian surface $A'$ with a polarization $L$ and an elliptic curve $B'$ such that $\ker L\cong B'[\ell^n]$, where $n=1$ or $n=2$.  By Lemma~\ref{lemma2}, one can choose $B'$ in its isogeny class in such a way that
there is an anti-isometry $B'[\ell^n]\cong \ker L$. Let $L_B$ be the polarization
on $B$ with kernel $B[\ell^n]$.
Now, by Lemma~\ref{glue}, the gluing of $L$ and $L_B$ gives an irreducible principal
polarization on a variety in the isogeny class of $A\times B$.
 The theorem is proved.
\end{proof}

\end{document}